\documentclass[10pt]{amsart}
\let\oldmarginpar\marginpar
\renewcommand\marginpar[1]{\-\oldmarginpar[\raggedleft\footnotesize #1]%
{\raggedright\footnotesize #1}}

\usepackage{mathptmx}
\usepackage{amsmath}
\usepackage{amscd}
\usepackage{amssymb}
\usepackage{amsthm}
\usepackage{xspace}
\usepackage[all,tips]{xy}
\usepackage[dvips]{graphicx}
\usepackage{verbatim}
\usepackage{syntonly}
\usepackage{hyperref}
\usepackage{arydshln}
\usepackage{lineno,color}

\theoremstyle{plain}
\newtheorem{thm}{Theorem}[section]
\newtheorem{cor}[thm]{Corollary}
\newtheorem{prop}[thm]{Proposition}
\newtheorem{lemma}[thm]{Lemma}

\theoremstyle{definition}

\newtheorem{defn}[thm]{Definition}

 \DeclareMathOperator{\Sol}{Sol}

\DeclareMathOperator{\lcm}{lcm}

\DeclareMathOperator{\D}{D}

\DeclareMathOperator{\Farb}{RF}

\DeclareMathOperator{\Exp}{Exp}

\newcommand{\bdef}{\overset{\text{def}}{=}}

\newcommand{\al}{\alpha}

\newcommand{\ga}{\gamma}

\newcommand{\innp}[1]{\left< #1 \right>}

\newcommand{\set}[1]{\left\{#1\right\}}

\newcommand{\pr}[1]{\left( #1 \right) }

\newcommand{\N}{\ensuremath{\mathbb{N}}}

\newcommand{\R}{\ensuremath{\mathbb{R}}}
\newcommand{\Z}{\ensuremath{\mathbb{Z}}}
\newcommand{\C}{\ensuremath{\mathbb{C}}}
\newcommand{\map}[3]{#1 : #2 \rightarrow #3}

\newcommand{\nsub}{\trianglelefteq}
\newtheoremstyle{TheoremNum}
{\topsep}{\topsep}
{\itshape}
{}
{\bfseries}
{.}
{ }
{\thmname{#1}\thmnote{ \bfseries #3}}
\theoremstyle{TheoremNum}
\newtheorem{thmn}{Theorem} 

\begin{document}


\title{\textbf{Residual finiteness and strict distortion of cyclic subgroups of solvable groups}}
\author{Mark Pengitore} 
\maketitle
\begin{abstract}
We provide polynomial lower bounds for residual finiteness of residually finite, finitely generated solvable groups that admit infinite order elements in the Fitting subgroup of strict distortion at least exponential. For this class of solvable groups which include polycyclic groups with a nontrivial exponential radical and the metabelian Baumslag-Solitar groups, we improve the lower bounds found in the literature. Additionally, for the class of residually finite, finitely generated solvable groups of infinite Pr\"{u}fer rank that satisfy the conditions of our theorem, we provide the first nontrivial lower bounds.
\end{abstract}


\section{Introduction}
A group $G$ is \textbf{residually finite} if for each nontrivial element $x \in G$, there exists a surjective group morphism $\varphi:G \to H$ to a finite group such that $\varphi(x) \neq 1$. In this case, we say that $H$ is a \textbf{finite witness of $x$}. When $G$ comes equipped with a finite generating subset $S$, we may quantify residual finiteness of $G$ with the function $\Farb_{G,S}(n)$ which was introduced in \cite{BouRabee10}. The value of $\Farb_{G,S}(n)$ is the maximum order of a finite group needed to witness a nontrivial element as one varies over nontrivial elements of word length at most $n$ with respect to the generating subset $S$. Since the dependence of $\Farb_{G,S}(n)$ on $S$ is mild, we will suppress the generating subset throughout the introduction.

Much of the previous work in the literature has been to compute $\Farb_{G}(n)$ for interesting classes of groups and to provide characterizations of various classes of groups based on the behavior of $\Farb_G(n)$ (see \cite{LLM} and the references therein).  Upper bounds have been provided for $\Farb_{G}(n)$ for many classes of groups such as the non-abelian free group of finite rank, finitely generated nilpotent groups, polycyclic groups, and more generally linear groups. On the other hand, providing lower bounds for $\Farb_{G}(n)$ tends be more difficult, and not as much work has been done. For instance, lower bounds have been provided for the non-abelian free group of finite rank, finitely generated nilpotent groups, and the first Grigorchuk group. Our main interest is in lower bounds for $\Farb_G(n)$ for residually finite, finitely generated solvable groups. \cite{BM11} gave the first nontrivial lower bound for residually finite, finitely generated solvable non-virtually nilpotent groups. For solvable groups $G$ of this class that are of finite Pr\"{u}fer rank, it was demonstrated that $n^{1/2m} \preceq \Farb_{G}(n)$ where $m$ is some natural number.  Additionally, for each recursive function $f: \N \to \N$, \cite{sapir_myasnikov} constructs a residually finite, finitely presentable solvable group $G$ of derived length $3$ such that $f(n) \preceq \Farb_{G}(n)$. Thus, for a general residually finite, finitely presentable solvable group, there is no possible class of functions that provide upper asymptotic bounds for residual finiteness. However, one can hope to provide lower asymptotic bounds for residual finiteness for a large class of residually finite, finitely generated solvable non-virtually nilpotent groups given certain assumptions. In this article, we provide an improved lower bound for $\Farb_{G}(n)$ when $G$ is a residually finite, finitely generated solvable group with elements in the Fitting subgroup which have strict distortion at least exponential. In particular, this article improves on the lower bounds found in the literature for finitely generated solvable groups of finite Pr\"{u}fer rank which satisfy our main theorem and gives the first nontrivial lower bound for the class of finitely generated solvable groups of infinite Pr\"{u}fer rank that satisfy the main theorem of this article.

Before proceeding to the main result of this paper, we introduce some notation and terminology. For two nondecreasing functions $f,g: \N \to \N$, we say that $f \preceq g$ if there exists a constant $C > 0$ such that $f(n) \leq C \: g(C \: n)$ for all $n$. We say that $f \approx g$ if $f \preceq g$ and $g \preceq f$. We say that an infinite order element $x$ in a group $G$ equipped with a finite generating subset $S$ is strictly $f$-distorted if there exist constants $C_1, C_2 > 0$ such that $C_1 |n| \leq f(\|x^n\|_S) \leq C_2 |n|$. For a group $G$, we denote $\ga_m(G)$ as the $m$-th step of the lower central series, and when $G$ is a finitely generated solvable group, we denote $\text{Fitt}(G)$ as the Fitting subgroup. For a nontrivial element $x \in \text{Fitt}(G)$, we say that $x$ has nilpotent depth $m$ if there exists a normal nilpotent subgroup $N$ of $G$ such that $\innp{x} \cap \ga_m(N) \neq \set{1}$.
\begin{thm}\label{lower_bound_rf}
	Let $G$ be a residually finite, finitely generated solvable group that is not virtually nilpotent, and let $f: \N \to \N$ be a nondecreasing function such that $2^n \preceq f(n) \preceq 2^{2^n}$. If $G$ admits an infinite order, strictly $f$-distorted element $x$ contained in $\text{Fitt}(G)$, then $$\log(f(n)) \preceq \Farb_{G}(n),$$ and if $x$ has nilpotent depth $m > 1$, then  $$(\log(f(n)))^{m + 1} \preceq \Farb_{G}(n).$$ 
In particular, $n \preceq \Farb_G(n)$, and if $x$ has nilpotent depth $m > 1$, then $n^{m+1} \preceq \Farb_G(n)$.
\end{thm}

Before proceeding to applications, we provide some justification for the various restrictions in our theorem. It was demonstrated in \cite{BouRabee10} that if $G$ is a virtually nilpotent group, then $\Farb_{G}(n) \preceq (\log(n))^k$ where $k$ is some natural number; consequently, we assume that our groups are not virtually nilpotent in order to have a new result. Since there exist examples of non-residually finite, finitely presentable solvable groups of derived length $3$, it is necessary to assume that our solvable groups are residually finite. We also note that if the solvable group $G$ admits an infinite order element, then $\log(n) \preceq \Farb_{G}(n)$ automatically. Thus, in order to have a nontrivial result, we must have that $2^n \preceq f(n)$. Finally, \cite[3.K1]{asymptotic_group} implies that if $2^{2^n} \preceq f(n)$, then there would be more than exponentially many points in $n$-balls of $G$ which is impossible. In particular, the strict distortion of any element is at most $2^{2^n}$ which implies the necessity of the upper bound for $f(n)$.

Our first application is to cocompact lattices in $\Sol$ and in solvable Lie groups of the form $\R^d \rtimes_M \R$ where $M$ is a positive definite matrix with all eigenvalues not equal to $1$. Cocompact lattices in these solvable Lie groups were the first class of polycyclic non-virtually nilpotent groups for which quasi-isometric rigidity results were announced, and in particular, these groups form an interesting class of residually finite, finitely generated solvable groups to study effective residually finiteness for. We have by the top of \cite[page 1684]{dymarz}  that if $G$ is a cocompact lattice in $\R^d \rtimes_M \R$, then each element of $G \cap \R^d$ is strictly exponentially distorted, and since $G \cap \R^d \cong \Z^d$, we are able to apply Theorem \ref{lower_bound_rf}. More generally, $\Sol$ and the solvable Lie groups $\R^d \rtimes_M \R$ are examples of what are known as nondegenerate, split abelian by abelian Lie groups for which similar statements can be made (see \cite{peng1, peng2} for a precise definition). Hence, we have the following corollary.
\begin{cor}
Let $G$ be a cocompact lattice in $\Sol$ or in $\R^d \rtimes_M \R$ where $M$ is a positive definite matrix with all eigenvalues not equal to $1$. Then $n \preceq \Farb_{G}(n)$. More generally, if $G$ is a cocompact lattice in a connected, simply connected, nondegenerate, split abelian by abelian solvable Lie group, then $n \preceq \Farb_{G}(n)$. 
\end{cor}

For a connected, simply connected, non-nilpotent, upper triangular Lie group $\textbf{G}$ with nilradical $\textbf{N}$ and exponential radical $\Exp(\textbf{G})$ (see \cite{dimension_asymptotic_cone_exponential, osin_exponential} for the definition of exponential radical), we have by using  \cite[Lemma 2.5]{dimension_asymptotic_cone_exponential} and basic facts found in \cite{auslander, Dekimpe, discrete_subgroups_rag} that if $G$ is a cocompact lattice in $\textbf{G}$, then there exists a natural number $m$ such that if $\Exp(\textbf{G}) \cap \ga_m(\textbf{N})$ is a nontrivial, connected, closed Lie subgroup of $\textbf{G}$ that admits $G \cap \Exp(\textbf{G}) \cap \ga_m(\textbf{N})$ as a cocompact lattice. In particular, we have that $G \cap \Exp(\textbf{G}) \cap \ga_m(\textbf{N})$ is a nontrivial subgroup of $G$ with elements of strict exponential distortion of nilpotent depth $m$. Hence, we have the following corollary.
\begin{cor}
Let $\textbf{G}$ be a connected, simply connected, non-nilpotent, triangular Lie group with nilradical $\textbf{N}$ and exponential radical $\Exp(\textbf{G})$, and suppose that $\textbf{G}$ admits a cocompact lattice $G$. If $m$ is the largest natural number such that $\Exp(\textbf{G}) \cap \ga_m(\textbf{N}) \neq \{1\}$, then $n \preceq \Farb_G(n)$, and if $m > 1$, then $n^{m+1} \preceq \Farb_{G}(n)$.
\end{cor}

We have the following corollary for the more general class of polycyclic groups that admit a strictly exponentially distorted element in the Fitting subgroup. We note for polycyclic groups that the Fitting subgroup is always a nontrivial nilpotent group.
\begin{cor}\label{polycyclic_cor}
If $G$ is an infinite polycyclic group with a nontrivial strictly exponentially distorted element $x \in \text{Fitt}(H)$, then $n \preceq \Farb_{G}(n)$. If $x \in \ga_m(\text{Fitt}(G))$ where $m > 1$, then $n^{m+1} \preceq \Farb_G(n)$.
\end{cor}

Our next application is to finitely generated metabelian groups which by Hall \cite[Theorem 1]{metabelian} are always residually finite. By \cite{metabelian_rem}, we have that every finitely generated metabelian group is linear over a finite product of fields, and in particular, if $G$ is virtually torsion free, then $G$ is linear over $\C$. Thus, \cite[Theorem 1.1]{BM12} implies that $\Farb_{G}(n) \preceq n^k$ for some natural number $k$.  Therefore, the best lower bound we can obtain for this class of residually finite, finitely generated solvable groups using Theorem \ref{lower_bound_rf} is polynomial which can be seen in the following corollary. For this corollary, we let $\text{BS}(k,m) = \innp{x,t \: | \: t x^k t^{-1} = x^m}$ be the Baumslag-Solitar group. We note that $\text{BS}(1,m)$ for $m > 1$ is a metabelian group that is not virtually nilpotent.
\begin{cor}
Let $G$ be a finitely generated metabelian group with a nontrivial element $x \in \text{Fitt}(G)$ that is strictly exponentially distorted. Then $n \preceq \Farb_G(n)$. In particular, $n \preceq \Farb_{\text{BS}(1,m)}(n)$ for $m > 1$.
\end{cor}

For this last corollary, we say that a finitely generated group $G$ has Pr\"{u}fer rank $r$ if every finitely generated subgroup of $G$ can be generated by at least $r$ elements and $r$ is the least such natural number. Otherwise, we say that $G$ has infinite Pr\"{u}fer rank.
\begin{cor}
Let $G$ be a residually finite, finitely generated solvable group of infinite Pr\"{u}fer rank with a nontrivial infinite order element $x \in \text{Fitt}(G)$ such that $x$ is at least strictly exponentially distorted. Then $n \preceq \Farb_{G}(n)$. If $x$ has nilpotent depth $m>1$, then $n^{m+1} \preceq \Farb_{G}(n)$.
\end{cor}

The proof of Theorem \ref{lower_bound_rf} proceeds by finding an infinite sequence of elements $\set{g_i}$ in $G$ whose minimal finite witness has order at least $\log(f(\|g_i\|))$. Similarly, when the distorted element has nilpotent depth $m > 1$ in $N$, we construct a sequence of elements $\{g_i\}$ such that the minimal finite witness of $g_i$ has order at least $(\log(f(\|g_i\|)))^{m+1}$. We first show that we may assume that the Fitting subgroup of any finite witness of the elements in consideration is a finite $p$-group for some prime $p$.  We then give conditions for the Fitting subgroup of any finite witness of a nontrivial element in $\text{Fitt}(G)$ to have order at least $p^{m+1}$ when the nilpotent depth of $x$ is $m > 1$. We finish by choosing an element $x$ that is strictly $f$-distorted and a sequence of integers $\set{\ell_i}$ using the Prime Number Theorem so that our desired sequence of elements is given by $\set{x^{\ell_i}}$.
\\

\noindent \textbf{Acknowledgments} We thank Ben McReynolds, Jean-Fran\c{c}ois Lafont, Yves de Cornulier, David Anderson, Tullia Dymarz, and Jim Cogdell for useful discussions regarding this article.

\section{Background}
\subsection{Notation} We denote $\lcm\set{1,\cdots,k}$ as the least common multiple of the natural numbers $1$ to $k$. We denote $1$ as the identity element of any group. For a finite group $G$, we denote $|G|$ as its cardinality. For $x \in G$, we denote $\text{Ord}_G(x)$ as the order of $x$ as an element of $G$. For a finitely generated group $G$ with a finite generating subset $S$, we denote $\|x\|_S$ as the word length of $x$ in $G$ with respect to the generating subset $S$, and we write $\|x\|$ when the generating subset is understood from context. We let $G^{(i)}$ be the $i$-th step of the derived series and $\ga_i(G)$ be the $i$-th step of the lower central series. For a finitely generated solvable group $G$, we denote $\text{Fitt}(G)$ as the Fitting subgroup of $G$.
\subsection{Separability}
We define the depth function $\map{\D_G}{G \: \backslash \set{1}}{\N \cup \{\infty\}}$ of $G$ to be given by
$$
\D_{G} (x) = \text{min}\set{|H| \: |  \: \varphi: G \rightarrow H, |H| < \infty, \text{ and } \varphi (x) \neq 1}
$$
with the understanding that $\D_{G}(x) = \infty$ if no such finite group $H$ exists.
\begin{defn}
	Let $G$ be a finitely generated group. We say that $G$ is \textbf{residually finite} if $\text{D}_G(x) < \infty$ for all $x \in G \:  \backslash \set{1}$.
\end{defn}
We define
$
\map{\Farb_{G,S}}{\mathbb{N}}{\mathbb{N}}
$
as
$$
\Farb_{G,S} (n) = \text{max}\set{\D_{G} (x) | \|x\|_S \leq n \text{ and } x \neq 1}.$$ For any two finite generating subsets $S_1$ and $S_2$, we have 
$ \Farb_{G,S_1} (n) \approx \Farb_{G,S_2} (n)$ (see \cite{BouRabee10}). Hence, we suppress reference to the generating subset in $\Farb_{G,S}(n)$ when it is clear from context.

\subsection{Solvable Groups}
For this next section, see \cite{robinson,Segal,tessera,wehrfritz} for a more thorough discussion about solvable groups.

We define the first term of the \textbf{derived series} of $G$ as $G^{(1)} \bdef G$, and for $i > 1$, we define inductively the $i$-th term of the derived series as $G^{(i)} \bdef [G^{(i-1)}, G^{(i-1)}]$. 

\begin{defn}
Let $G$ be a finitely generated group. We say that $G$ is a \textbf{solvable group of derived length $s$} if $s$ is the smallest natural number such that $G^{(s+1)} = \set{1}$, and when the derived length is unspecified, we just say that $G$ is a solvable group. We say that $G$ is a \textbf{nilpotent group of step length $c$} if $c$ is the smallest natural number such that $\ga_{c+1}(G) = \set{1}$. As before, if the step length is unspecified, then we say that $G$ is a nilpotent group. We say a finitely generated solvable group is a \textbf{polycyclic group} if every subgroup is finitely generated.
\end{defn}

We now introduce the Fitting subgroup of a finitely generated solvable group which will be an essential tool in the proof of Theorem \ref{lower_bound_rf}.
\begin{defn}
Let $G$ be a finitely generated solvable group. The \textbf{Fitting subgroup of $G$}, denoted $\text{Fitt}(G)$, is the characteristic subgroup generated by all normal nilpotent subgroups of $G$
\end{defn}
Since the last nontrivial term of the derived series of any solvable group $G$ is a normal abelian subgroup, we have that $\text{Fitt}(G) \neq \set{1}$.  When $G$ is a finite solvable group or more generally a polycyclic group, we have that $\text{Fitt}(G)$ is a finitely generated nilpotent group. However, for a general finitely generated solvable group, it is not necessarily the case that either $\text{Fitt}(G)$ is nilpotent or even a finitely generated group.

We finish this section with the follow definition.
\begin{defn}
Let $G$ be a finitely generated solvable group with a nontrivial element $x \in \text{Fitt}(G)$. We say that $x$ has \textbf{nilpotent depth} $m$ if there exists a normal nilpotent subgroup $N$ such that $m$ is the largest natural number where $\innp{x} \cap \in \ga_m(N) \neq \set{1}$. 
\end{defn}

\section{Finite quotients of solvable groups}
We start with the following lemma that relates the cardinality of a finite $p$-group with its step length.
\begin{lemma}\label{p_dim_nilpotent_length}
	Let $p$ be a prime number. If $Q$ is an abelian finite $p$-group, then $|Q| \geq p$. If $Q$ is a finite $p$-group of step length $c>1$, then $|Q| \geq p^{c+1}$.
\end{lemma}
\begin{proof}
	Since the first statement is clear, we may assume that $Q$ has step length $c > 1$. We prove the second statement by induction on step length, and for the base case, we assume that $Q$ has step length $2$. There exist nontrivial elements $x,y \in Q$ such that $[x,y] \neq 1$, and since $Q$ has step length $2$, we have that $[x,y] \in [Q,Q] \leq Z(Q)$. For the subgroup $H = \innp{x,y,[x,y]}$, we have that each element  can be written uniquely as $x^t \: y^s \: [x,y]^\ell$ for natural numbers $0 \leq t < \text{Ord}_Q(x)$, $0 \leq s < \text{Ord}_Q(y)$, and $0 \leq \ell < \text{Ord}_Q([x,y])$. Therefore, we have that 
	$
	|H| = \text{Ord}_Q(x) \cdot \text{Ord}_Q(y) \cdot \text{Ord}_Q([x,y]) \geq p^3.
	$
	Thus, $|Q| \geq p^3$, and since $|H| \mid |Q|$, we have that $|Q| \geq p^3$.
	
	We now assume that $Q$ has step length $c > 2$. By induction, we have that $|Q/ Q_c| \geq p^c$, and since $Q_c$ is abelian, we have that $|Q_c| \geq p$. In particular, $|Q| = |Q / Q_c| \cdot |Q_c| \geq p^{c+1}.$
\end{proof}

Let $G$ be a finitely generated solvable group with a nontrivial element $x \in \text{Fitt}(G)$. The next proposition implies that if given a surjective group morphism $\varphi: G\rightarrow H$ where $\varphi(x) \neq 1$, then we may assume $\text{Fitt}(H)$ is a finite $p$-group for some prime $p$.

\begin{prop}\label{special_p_by_abelian_quotient}
	Let $G$ be a residually finite, finitely generated solvable group with a nontrivial element $g \in \text{Fitt}(G)$. Let $\varphi:G \rightarrow H$ be a surjective group morphism to a finite group where $\varphi(x) \neq 1$. Then there exists a prime number $p$ and a normal subgroup $K \nsub H$ such that $\text{Fitt}(H/K)$ is a finite $p$-group where $\varphi(x) \notin K$.
\end{prop}
\begin{proof}
	 Since $H$ is solvable, there exists a subnormal series of subgroups $\set{1} = H_0 \leq H_1 \leq \cdots \leq H_r = H$, known as a composition series, such that $|H_{i+1}/H_i| = p_i$ where $p_i$ is some prime. We proceed by induction on the length of the composition series which is given by $r$, and since the base case is clear, we may assume that $r > 1$.

Given that $x \in \text{Fitt}(G)$, there exists finitely many normal nilpotent subgroup $\set{N_i}_{i=1}^t$ such that $x \in N$ where $N = N_1 \cdot N_2 \cdots N_t$.  Thus, \cite[Theorem 2.5]{Hall_notes} implies that $N$ is a normal nilpotent subgroup of $G$.  Since $\varphi(g) \neq 1$, we have that $\varphi(N)$ is a nontrivial, normal nilpotent subgroup of $H$. Therefore, $\varphi(N)$ is a subgroup of $\text{Fitt}(H)$, and thus, $\varphi(x) \in \text{Fitt}(H)$. If $\text{Fitt}(H)$ is a finite $q$-group for some prime $q$, then we are done. Otherwise, we may assume that $|\text{Fitt}(H)| = \prod_{s=1}^k q_s^{t_s}$ where each $q_i$ is a prime. Since \cite[Theorem 2.7]{Hall_notes} implies that $\text{Fitt}(H) = \prod_{i=1}^{k}Q_i$ where $|Q_i| = q_i^{t_i}$, we may write $\varphi(x) = \pr{a_1,\cdots, a_k}$ where $a_i \in Q_i$, and given that $\varphi(x) \neq 1$, there exists some $i_0$ such that $a_{i_0} \neq 1$. Given that $K = \prod_{i=1, i \neq i_0}^k Q_i$ is a characteristic subgroup of $\text{Fitt}(H)$, we have that $K$ is a normal subgroup of $H$ where $\varphi(x) \notin K$. We note that $H/K$ has composition length strictly less than that of $H$; thus, by induction, we have that there exists a normal subgroup $W/K \nsub H/K$ such that $\text{Fitt}(H/W)$ is a finite $q$-group for some prime $q$ and where $\varphi(x) \notin W$ as desired. \end{proof}

\section{Strict distortion of cyclic subgroups in solvable groups}
Throughout this section, $f:\N \to \N$ will be a nondecreasing function where $f(n) \preceq 2^{2^n}$.
\begin{defn}
	Let $G$ be a finitely generated group with a finite generating subset $S$. We say that an infinite order element $x \in G$ is \textbf{strictly $f$-distorted} if there exists constants $C_1,C_2 >0$ such that
	$$
	C_1 \: |n| \leq f(\|x^n\|_S) \leq C_2 \: |n|
$$
for all $n > 0$. If $f(n) = 2^n$, we say that $x$ is \textbf{strictly exponentially distorted}.  When $f$ is linear, we say that $x$ is \textbf{undistorted.}
\end{defn}

The following simple lemma relates the strict distortion of an infinite order element with the strict distortion of proper powers of that element.
\begin{lemma}\label{distorted_powers}
Let $G$ be a finitely generated group with a nontrivial infinite order element $x \in G$ that is strictly $f$-distorted. For all $k > 1$, we have that $x^k$ is a strictly $f$-distorted element of $G$.
\end{lemma}
\begin{proof}
Let $S$ be a finite generating subset of $G$. There exists constants $C_1, C_2 > 0$ such that $C_1 \: |n| \leq f(\|x^n\|_S) \leq C_2 \: |n|.$ Thus, we have $C_1 \:  k \: |n|  \leq f(\|(x^k)^n\|_S) \leq C_2 \: k \: |n|$.
\end{proof}

It is important to note that if $G$ is a finitely generated solvable group of exponential growth, then it not necessarily the case that $G$ has any distorted elements. Indeed, \cite[Example 7.1]{conner} constructs a group of the form $\Z^4 \rtimes \Z$ which has exponential growth but every element is undistorted. 

For a finitely generated solvable group $G$ with a nontrivial element $x \in \text{Fitt}(G)$, this last proposition gives lower bounds for the size of any finite $p$-witness of $x$ in terms of the nilpotent depth of $x$.
\begin{prop}\label{p_dim_lower_bound}
Let $G$ be a residually finite, finitely generated solvable group with a nontrivial element $x  \in \text{Fitt}(G)$ of nilpotent depth $m$. If $\varphi:G\to H$ is a finite witness for $x$ where $\text{Fitt}(H)$ is a finite $p$-group, then $|H| \geq p$. Moreover, if $m > 1$, then $|H| \geq p^{m + 1}$.
\end{prop}
\begin{proof}
Since the first statement is clear, we may assume that $m > 1$. By definition, there exists a normal nilpotent subgroup $N$ of $G$ such that $x \in \ga_m(N)$. Since $\varphi(N)$ is nontrivial normal nilpotent subgroup, we have that $\varphi(N) \leq \text{Fitt}(H)$. We claim that $\text{Fitt}(H)$ has step length at least $m$, and for a contradiction, suppose that $\ga_m(\text{Fitt}(H)) = \set{1}$. Since $\varphi(\ga_m(N)) \leq \ga_m(\text{Fitt}(H))$, we must have that $\varphi(x) = 1$ which is a contradiction. Hence, we must have that $\text{Fitt}(H)$ has step length $\ell$ where $\ell \geq m$. By Lemma \ref{p_dim_nilpotent_length}, we have that $|\text{Fitt}(H)| \geq p^{m+1}$. Therefore, $|H| \geq p^{m+1}$. \end{proof}

\section{Proof of Main Theorem}
For the convenience of the reader, we restate Theorem \ref{lower_bound_rf}.
\begin{thmn}[\ref{lower_bound_rf}]
	Let $G$ be a residually finite, finitely generated solvable group that is not virtually nilpotent, and let $f: \N \to \N$ be a nondecreasing function such that $2^n \preceq f(n) \preceq 2^{2^n}$. If $G$ admits an infinite order, strictly $f$-distorted element $x$ contained in $\text{Fitt}(G)$, then $$\log(f(n)) \preceq \Farb_{G}(n),$$ and if $x$ has nilpotent depth $m > 1$, then  $$(\log(f(n)))^{m + 1} \preceq \Farb_{G}(n).$$ 
In particular, $n \preceq \Farb_G(n)$, and if $x$ has nilpotent depth $m > 1$, then $n^{m+1} \preceq \Farb_G(n)$.
\end{thmn}

\begin{proof}
We first demonstrate that $\log(f(n)) \preceq \Farb_{G}(n)$. Letting $\set{p_i}_{i=1}^{\infty}$ be an enumeration of the primes, we consider the sequence of natural numbers given by $\al_i = \lcm \set{1,\cdots, p_i - 1}$. We claim that $\log(f(\|x^{\al_i}\|_S)) \leq \D_{G}(x^{\al_i})$. By the Prime Number Theorem, we have that $\log(\al_i) \approx p_i$, and by definition, we have that $f(\|x^{\al_i}\|_S) \approx \al_i$. Therefore, $\log(f(\|x^{\al_i}\|_S)) \approx p_i$.  Letting $\varphi: G\to H$ be a surjective group morphism to a finite group where $|H| < p_i$, we have that $\text{Ord}_H(\varphi(x)) < p_i$. Thus, we must have that $\text{Ord}_H(\varphi(x)) \mid \al_i$ which implies that $\varphi(x^{\al_i}) = 1$. Hence, we have the first statement of our theorem.

Now let us assume that $x$ has nilpotent depth $m > 1$. Lemma \ref{distorted_powers} allows us to assume that $x \in \ga_m(N)$ for some normal nilpotent subgroup $N$ of $G$. Denoting an enumeration of the primes as $\set{p_i}_{i=1}^\infty$, we define $\al_i = \lcm\set{1,\cdots, p_i - 1}^{m+2}.$ We will demonstrate that 
	$
	(\log(f\pr{\|x^{\al_i}\|_S})^{m + 1} \leq \D_G(x^{\al_i}).
	$
	By the Prime Number Theorem, we have that $\log(\al_i) \approx p_i$, and by definition, we have $f\pr{\|x^{\al_i}\|_S} \approx \al_i$. Therefore, $\log(f(\|x^{\al_i}\|_S)) \approx p_i$. Hence, we must show that $p_i^{m+ 1} \leq \D_G(x^{\al_i})$.
	
	Let $\map{\varphi}{G}{H}$ be a surjective group morphism to a finite group where $|H| < p_i^{m+1}$. Proposition \ref{special_p_by_abelian_quotient} implies that we may assume that $\text{Fitt}(H)$ is a finite $q$-group for some prime $q$. Since $|H| < p_i^{m+1}$, we have that $|\text{Fitt}(H)| < p_i^{m+1}$ leading to a number of cases.
	
Suppose first that $|\text{Fitt}(H)|< p_i$. By construction, $|\text{Fitt}(H)| \mid \al_i$, and since the order of an element divides the order of the group, we have that $\text{Ord}_H(\varphi(x)) \mid \al_i$. In particular, we have that $\varphi(x^{\al_j}) = 1$ in $H$.
	
Now suppose that $q < p_i$ and that $p_i < |\text{Fitt}(H)|< p_i^{m + 1}$. There exists some natural number $v$ such that $q^v < p_i < q^{v+1}$, and thus,
	$
	q^{v \: (m+1)} < p_i^{m+1} < q^{(v+1) (m+1)}.
	$
	Therefore, we have that $|\text{Fitt}(H)| = q^{v t+ r}$ where $t \leq m + 1$ and $0 \leq r < v$. Given that $q^v < p_i$, we have that $q^v \mid \lcm \set{1, \cdots, p_i - 1}$; in particular, it follows that
	$
	q^{v \: t} \mid \pr{\lcm \set{1, \cdots, p_i - 1}}^{m+1}.
	$
	Since $r < v$, we have that $q^r < p_i$. Hence, it follows that $q^r \mid \lcm \set{1, \cdots, p_i - 1}$. Therefore, $|\text{Fitt}(H)| \mid \al_i$, and since the order of $\varphi(x)$ divides $|\text{Fitt}(H)|$, we must have that $\varphi(x^{\al_i}) = 1$.

Finally, we assume that $q \geq p_i$. If $\varphi(x) = 1$, there is nothing to prove. Therefore, we may assume that $\varphi(x) \neq 1$, and by Proposition \ref{p_dim_lower_bound}, we have that $|\text{Fitt}(H)| \geq q^{m+1}$. Hence, we have that either $\varphi(x) = 1$ or $|H| \geq p_i^{m+1}$. In particular, we may ignore this possibility.
	
	Since each possibility has been covered, we must have that $p_{i}^{m+1}\leq \D_{G}(x^{\al_i})$. Therefore, it follows that $(\log(f(\|x_i^{\al_i}\|_S))^{m+1} \preceq \D_{G}(x_i^{\al_i})$. Thus, $(\log(f(n)))^{m+1}  \preceq \Farb_G(n)$.
\end{proof}
\bibliography{bib2}

\begin{thebibliography}{10}

\bibitem{auslander}
Louis Auslander.
\newblock An exposition of the structure of solvmanifolds. {I}. {A}lgebraic
  theory.
\newblock {\em Bull. Amer. Math. Soc.}, 79(2):227--261, 1973.

\bibitem{BM11}
K.~Bou-Rabee and D.~B. McReynolds.
\newblock Asymptotic growth and least common multiples in groups.
\newblock {\em Bull. Lond. Math. Soc.}, 43(6):1059--1068, 2011.

\bibitem{BouRabee10}
Khalid Bou-Rabee.
\newblock Quantifying residual finiteness.
\newblock {\em J. Algebra}, 323(3):729--737, 2010.

\bibitem{BM12}
Khalid Bou-Rabee and D.~B. McReynolds.
\newblock Extremal behavior of divisibility functions.
\newblock {\em Geom. Dedicata}, 175:407--415, 2015.

\bibitem{conner}
Gregory~R. Conner.
\newblock Discreteness properties of translation numbers in solvable groups.
\newblock {\em J. Group Theory}, 3(1):77--94, 2000.

\bibitem{dimension_asymptotic_cone_exponential}
Yves de~Cornulier.
\newblock Dimension of asymptotic cones of {L}ie groups.
\newblock {\em J. Topol.}, 1(2):342--361, 2008.

\bibitem{Dekimpe}
Karel Dekimpe.
\newblock Almost-bieberbach groups: affine and polynomial structures.
\newblock 1996.

\bibitem{dymarz}
Tullia Dymarz.
\newblock Large scale geometry of certain solvable groups.
\newblock {\em Geom. Funct. Anal.}, 19(6):1650--1687, 2010.

\bibitem{asymptotic_group}
M.~Gromov.
\newblock {\em Asymptotic invariants of infinite groups}.
\newblock London Math. Soc. Lecture Note Ser., 182. Cambridge Univ. Press,
  Cambridge, 1993.

\bibitem{metabelian}
P.~Hall.
\newblock On the finiteness of certain soluble groups.
\newblock {\em Proc. London Math. Soc. (3)}, 9:595--622, 1959.

\bibitem{Hall_notes}
Philip Hall.
\newblock {\em The Edmonton notes on nilpotent groups}.
\newblock Queen Mary College Mathematics Notes. Mathematics Department, Queen
  Mary College, London, 1969.

\bibitem{sapir_myasnikov}
Olga Kharlampovich, Alexei Myasnikov, and Mark Sapir.
\newblock Algorithmically complex residually finite groups.
\newblock {\em Bull. Math. Sci.}, 7(2):309--352, 2017.

\bibitem{LLM}
Sean Lawton, Larsen Louder, and D.~B. McReynolds.
\newblock Decision problems, complexity, traces, and representations.
\newblock {\em Groups Geom. Dyn.}, 11(1):165--188, 2017.

\bibitem{robinson}
John~C. Lennox and Derek J.~S. Robinson.
\newblock {\em The theory of infinite soluble groups}.
\newblock Oxford Mathematical Monographs. The Clarendon Press, Oxford
  University Press, Oxford, 2004.

\bibitem{osin_exponential}
D.~V. Osin.
\newblock Exponential radicals of solvable {L}ie groups.
\newblock {\em J. Algebra}, 248(2):790--805, 2002.

\bibitem{peng1}
Irine Peng.
\newblock Coarse differentiation and quasi-isometries of a class of solvable
  {L}ie groups {I}.
\newblock {\em Geom. Topol.}, 15(4):1883--1925, 2011.

\bibitem{peng2}
Irine Peng.
\newblock Coarse differentiation and quasi-isometries of a class of solvable
  {L}ie groups {II}.
\newblock {\em Geom. Topol.}, 15(4):1927--1981, 2011.

\bibitem{discrete_subgroups_rag}
M.~S. Raghunathan.
\newblock Discrete subgroups of {L}ie groups.
\newblock {\em Math. Student}, (Special Centenary Volume):59--70 (2008), 2007.

\bibitem{metabelian_rem}
V.~N. Remeslennikov.
\newblock Representation of finitely generated metabelian groups by matrices.
\newblock {\em Algebra i Logika}, 8:72--75, 1969.

\bibitem{Segal}
Daniel Segal.
\newblock {\em Polycyclic groups}.
\newblock Cambridge Tracts in Mathematics, 82. Cambridge University, 1983.

\bibitem{tessera}
Romain Tessera.
\newblock The large-scale geometry of locally compact solvable groups.
\newblock {\em Internat. J. Algebra Comput.}, 26(2):249--281, 2016.

\bibitem{wehrfritz}
B.~A.~F. Wehrfritz.
\newblock {\em Infinite linear groups}.
\newblock Queen Mary College Mathematical Notes. Queen Mary College, Department
  of Pure Mathematics, London, 1969.

\end{thebibliography}
\bibliographystyle{plain}
\end{document}